\documentclass{article}

\usepackage{amsmath}
\usepackage{amssymb}
\usepackage{amsthm}
\usepackage{graphicx}
\usepackage{texdraw}
\newtheoremstyle{plain2}{\topsep}{\topsep}%
     {\itshape}
     {}
     {\bfseries}
     {.}
     {.5em}
     {\thmnumber{(#2)}\thmname{ #1}\thmnote{ #3}}

\theoremstyle{plain2}
\newtheorem{teo}{Theorem}[section]
\newtheorem{prop}[teo]{Proposition}
\newtheorem{coro}[teo]{Corollary}
\newtheorem{lemma}[teo]{Lemma}

\newtheoremstyle{definition2}{\topsep}{\topsep}%
     {}
     {}
     {\bfseries}
     {.}
     {.5em}
     {\thmnumber{(#2)}\thmname{ #1}\thmnote{ #3}}

\theoremstyle{definition2}
\newtheorem{example}[teo]{Example}
\newtheorem{rem}[teo]{Remark}

\def\N{\mathbb{N}}

\def\R{\mathbb{R}}
\def\C{\mathbb{C}}

\def\a{\alpha}
\def\b{\beta}

\def\d{\delta}
\def\ep{\varepsilon}
\def\eps{\varepsilon}
\def\f{\varphi}
\def\l{\lambda}

\def\s{\sigma}

\def\D{\Delta}
\def\L{\Lambda}

\def\Neh{\mathcal{N}}
\def\Kah{\mathcal{K}}
\def\Xh{\mathcal{X}}

\def\supp{\mathrm{supp\,}}

\def\spc{H^1_{r}(\R^N)}

\title{\sc Multipulse phases in $k$--mixtures of Bose--Einstein
condensates
\footnote{Work partially supported by MIUR, Project ``Metodi
Variazionali ed Equazioni Differenziali Non Lineari'' }}

\author{%
Susanna Terracini\\
Dipartimento di Matematica e Applicazioni\\ Universit\`a degli Studi
di Milano-Bicocca\\ Via Bicocca degli Arcimboldi, 8\\ 20126 Milano,
Italy\\\texttt{susanna.terracini@unimib.it}  \and Gianmaria Verzini\\
Dipartimento di Matematica \\ Politecnico di Milano\\ Piazza
Leonardo da Vinci, 32\\  20133 Milano, Italy\\
\texttt{gianmaria.verzini@polimi.it}}

\date{}

\begin{document}

\maketitle

\begin{abstract}
For the system
\[
-\D U_i+ U_i=U_i^3-\b U_i\sum_{j\neq i}U_j^2,\qquad i=1,\dots,k,
\]
(with $k\geq3$) we prove the existence, for $\b$ large, of positive
radial solutions on $\R^N$. We show that, as $\b\to+\infty$, the
profile of each component $U_i$ separates, in many pulses, from the
others. Moreover, we can prescribe the location of such pulses in
terms of the oscillations of the changing--sign solutions of the
scalar equation $-\D W+ W=W^3$. Within an Hartree--Fock
approximation, this provides a theoretical indication of phase
separation into many nodal domains for the $k$--mixtures of
Bose--Einstein condensates.
\end{abstract}

\section{Introduction}

In this paper we seek radial solutions to the  system of elliptic equations
\begin{equation}\label{eq:sys}
\left\{
 \begin{array}{l}
  -\D U_i+ U_i=U_i^3-\b U_i\sum_{j\neq i}U_j^2,\qquad i=1,\dots,k\smallskip\\
  U_i\in H^1(\R^N),\quad U_i>0,
 \end{array}
\right.
\end{equation}
with $N=2,3$, $k\geq 3$, and $\b$ (positive and) large, in
connection with the changing--sign solutions of the scalar equation
\begin{equation}\label{eq:sing}
-\D W+W=W^3,\qquad W\in H^1(\R^N).
\end{equation}
It is well known (see, for instance, \cite{s,jk}) that this equation
admits infinitely many nodal solutions. More precisely, following
Bartsch and Willem \cite{bw}, for any $h\in\N$ equation
\eqref{eq:sing} possesses radial solutions with exactly $h-1$
changes of sign, that
is $h$ nodal components (``bumps''), with a variational characterization.\\
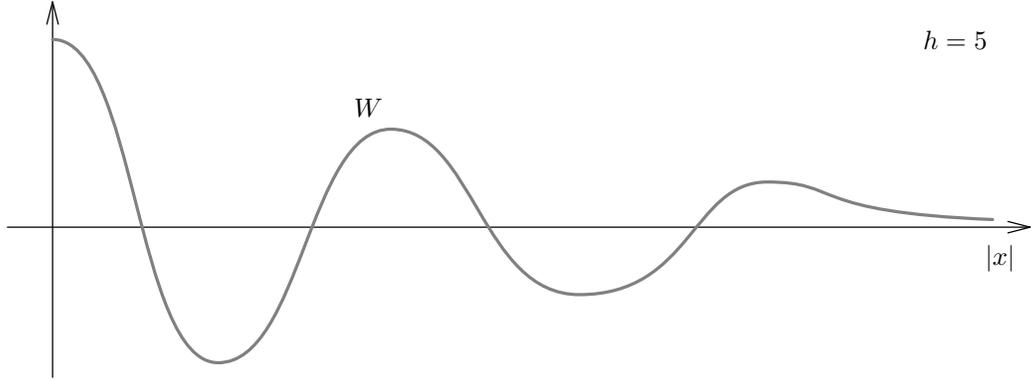
\begin{center}
\begin{texdraw}

\drawdim cm  \setunitscale 1.0

\linewd 0.02 \setgray 0 \lpatt()

\move (-0.6 0) \arrowheadtype t:V \arrowheadsize l:0.3 w:0.15 \avec
(13 0)

\move (0 -2) \avec (0 3) \arrowheadtype t:V

\linewd 0.04

\setgray 0.5 \lpatt () \move (0 2.5)
\clvec(1.1 2.5)(1.1 -1.8)(2.2 -1.8)
\clvec(3.3 -1.8)(3.4 1.3)(4.5 1.3)
\clvec (5.6 1.3)(5.7 -0.9)(7 -0.9)
\clvec(8.5 -0.9)(8.5 0.6)(9.5 0.6)
\clvec(10.5 0.6)(10 0.2)(12.5 0.1)
\textref h:C v:C \htext (12.6 -0.4) {$|x|$} \htext (4.2 1.6) {$W$}
\htext (12 2.5) {$h=5$}
\end{texdraw}
\end{center}
In the recent paper \cite{ww}, Wei and Weth have shown that, in the
case of $k=2$ components, there are solutions $(U_1,U_2)$ such that
the difference $U_1-U_2$, for large values of $\b$, approaches some
sign--changing solution $W$ of \eqref{eq:sing}. Hence, one can
prescribe the limit shape of $U_1$ and $U_2$ as $W^+$ and $W^-$:
this means that each $U_i$ can be seen as the sum of pulses, each
converging to one of the bumps of $|W|$.
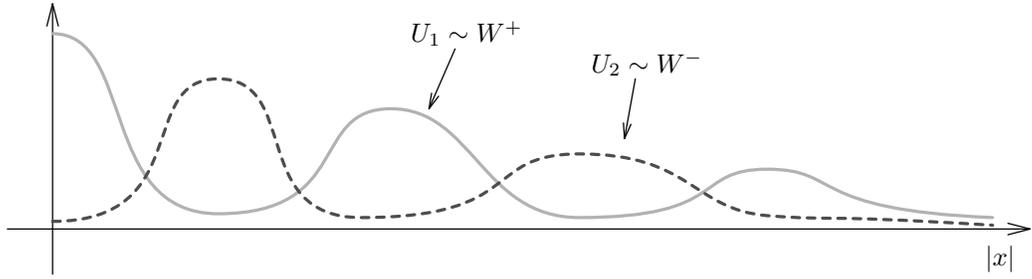
\begin{center}
\begin{texdraw}

\drawdim cm  \setunitscale 1.0

\linewd 0.02 \setgray 0 \lpatt()

\move (-0.6 0) \arrowheadtype t:V \arrowheadsize l:0.3 w:0.15 \avec
(13 0)

\move (0 -0.6) \avec (0 3) \arrowheadtype t:V
\linewd 0.02 \setgray 0 \lpatt()

\move (5.35 2.4) \arrowheadtype t:V \arrowheadsize l:0.2 w:0.1 \avec
(5 1.6)

\move (7.75 2) \arrowheadtype t:V \arrowheadsize l:0.2 w:0.1 \avec
(7.6 1.2)
\linewd 0.04 \setgray 0.7 \lpatt () \move (0 2.6)
\clvec(1.1 2.6)(0.6 0.2)(2.2 0.2)
\clvec(4 0.2)(3.4 1.6)(4.5 1.6)
\clvec (5.6 1.6)(5.7 0.15)(7 0.15)
\clvec(9 0.15)(8.5 0.8)(9.5 0.8)
\clvec(10.5 0.8)(10 0.25)(12.5 0.15)
\setgray 0.3 \lpatt (0.1 0.1) \move (0 0.1)
\clvec(1.9 0.1)(1.1 2)(2.2 2)
\clvec(3.3 2)(2.6 0.15)(4.1 0.15)
\clvec (6.3 0.15)(5.7 1)(7 1)
\clvec(8.5 1)(8.5 0.26)(9.5 0.18)
\clvec(10.5 0.1)(10 0.2)(12.5 0.05)
\textref h:C v:C \htext (12.6 -0.4) {$|x|$} \htext (5.5 2.6)
{$U_1\sim W^+$} \htext (7.9 2.2) {$U_2\sim W^-$}

\end{texdraw}
\end{center}
In the present paper we extend this result to the case of an
arbitrary number of components $k\geq3$, proving the existence of
solutions to \eqref{eq:sys} with the property that, for $\b$ large,
each component $U_i$ is near the sum of some non--consecutive bumps
of $|W|$ (see Theorems \ref{teo:main} and \ref{teo:main2}).
\begin{center}
\begin{texdraw}

\drawdim cm  \setunitscale 1.0

\linewd 0.02 \setgray 0 \lpatt()

\move (-0.6 0) \arrowheadtype t:V \arrowheadsize l:0.3 w:0.15 \avec
(13 0)

\move (0 -0.6) \avec (0 3) \arrowheadtype t:V

\linewd 0.04
\setgray 0.85 \lpatt () \move (0 2.4)
\clvec(1.3 2.4)(0.6 0.15)(2.2 0.15)
\clvec(4 0.15)(3.4 1.7)(4.5 1.7)
\clvec (5.6 1.7)(5.7 0.35)(7.2 0.25)
\clvec(7.7 0.15)(10 0.05)(12.5 0.04)
\setgray 0.15 \lpatt (0.1 0.1) \move (0 0.2)
\clvec(1.9 0.2)(1.1 2)(2.2 2)
\clvec(3.8 2)(2.6 0.15)(6.1 0.15)
\clvec (9.3 0.15)(8 0.8)(9.5 0.8)
\clvec(10.5 0.8)(10 0.25)(12.5 0.15)
\setgray 0.5 \lpatt (0.2 0.1) \move (0 0.1)
\clvec(3.9 0.1)(4.5 0.1)(5.5 0.3)
\clvec(6 0.4)(6 1)(7 1)
\clvec(8.5 1)(8.5 0.26)(9.5 0.18)
\clvec(10.5 0.1)(10 0.2)(12.5 0.09)
\textref h:C v:C \htext (12.6 -0.4) {$|x|$} \htext (1.5 2.7)
{$U_1\sim w_1+w_3$} \htext (4.5 2.2) {$U_2\sim w_2+w_5$} \htext (8.4
1.9) {$U_3\sim w_4$}
\linewd 0.02 \setgray 0 \lpatt()

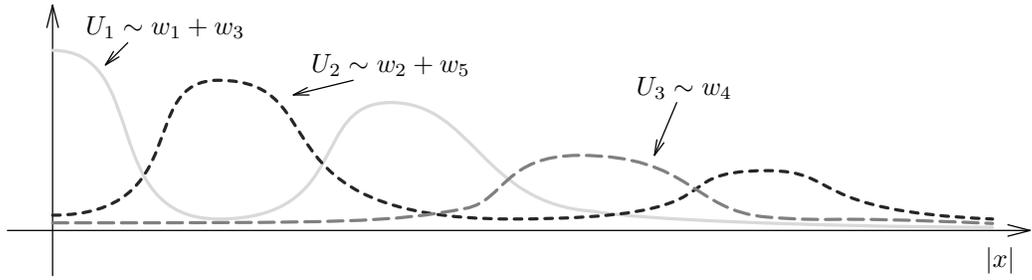
\move (1 2.5) \arrowheadtype t:V \arrowheadsize l:0.2 w:0.1 \avec
(0.7 2.2)

\move (4 2) \arrowheadtype t:V \arrowheadsize l:0.2 w:0.1 \avec (3.2
1.8)

\move (8.3 1.7) \arrowheadtype t:V \arrowheadsize l:0.2 w:0.1 \avec
(8 1)

\end{texdraw}
\end{center}
Furthermore, we can prescribe the correspondence between such bumps
of $|W|$ and the index $i$ of the component $U_i$ (see Example
\ref{example}). This, compared with the case $k=2$, provides a much
richer structure of the solution set for \eqref{eq:sys}. This goal
will be achieved by a suitable construction inspired by the extended
Nehari method (see \cite{neh}) developed in \cite{ctv2}.

System \eqref{eq:sys} arises in the study of solitary wave solutions
of systems of $k\geq3$ coupled nonlinear Schr\"odinger equations,
known in the literature as Gross--Pitaevskii equations:
\[
\left\{
 \begin{array}{l}
  \displaystyle -\imath \partial_t(\phi_i)=\Delta \phi_i-V_i(x)\phi_i+ \mu_i|\phi_i|^2\phi_i-
  \sum_{j\neq i}\b_{ij}|\phi_j|^2\phi_i,\qquad i=1,\dots,k\smallskip\\
  \phi_i \in H^1(\R^N;\C),\qquad N=1,2,3.
 \end{array}
\right.
\]
This system has been proposed as a mathematical model for
multispecies Bose--Einstein condensation in $k$ different hyperfine
spin states (see \cite{clll} and references therein); such a
condensation has been experimentally observed in the triplet states
(see \cite{nature}). Here the complex valued functions $\phi_i$'s
are the wave functions of the $i$--th condensate, the functions
$V_i$'s represent the trapping magnetic potentials, and the positive
constants $\mu_i$'s and $\b_{ij}$'s are the intraspecies and the
interspecies scattering lengths, respectively. With this choice the
interactions between like particles are attractive, while the
interactions between the unlike ones are repulsive; we shall assume
that $\b_{ij}=\b_{ji}$, which gives the system a gradient structure.
To obtain solitary wave solutions we set
\[
\phi_i(t,x)=e^{-\imath\lambda_i t}U_i(x),
\]
obtaining that the real functions $U_i$'s satisfy
\begin{equation}\label{eq:sys_comp}
\left\{
 \begin{array}{l}
  \displaystyle -\Delta U_i+\left[V_i(x)+\lambda_i\right]U_i= \mu_iU_i^3-
  \sum_{j\neq i}\b_{ij}U_j^2U_i,\qquad i=1,\dots,k\smallskip\\
  U_i \in H^1(\R^N).
 \end{array}
\right.
\end{equation}
For the sake of simplicity we assume $V_i(x)\equiv0$,
$\lambda_i=\mu_i=1$ and $\b_{ij}=\b$, for every $i$ and $j$, and
$N=2,3$, even though our method works also in more general cases,
see Remark \ref{rem:finale} at the end of the paper. With this
choice, system \eqref{eq:sys_comp} becomes system \eqref{eq:sys}.

For a fixed $k$, as the interspecific competition goes to infinity,
the wave amplitudes $U_i$'s segregate, that is, their supports tend
to be disjoint. This phenomenon, called ``phase separation'', has
been studied, starting from \cite{ctv,ctv2}, in the case of
$\mu_i>0$ and in \cite{clll} in the case $\mu_i<0$, for least energy
solutions in non necessarily symmetric bounded domains. Of course,
the number of connected domains of segregation is at least the
number of different phases surviving in the limit. For the minimal
solutions, the limiting states have \emph{connected}
supports\footnote{This is rigorously proven in \cite{ctv2}, while it
results from numerical evidence in \cite{clll}.}. This is not
necessarily the case for solutions which are not characterized as
ground states. This is indeed  what we show in the present paper,
proving the existence of solutions converging to limiting states
which supports have a large number of connected components. In this
way we obtain a large number of connected domains of segregation
with a few phases. Taking the limiting supports as unknown, this can
be seen as a free boundary problem. The local analysis of the
interfaces and the asymptotic analysis, as the interspecific
scattering length grows to infinity has been carried in \cite{ctv2}
for the minimal solutions.

In the recent literature, systems of type \eqref{eq:sys} have been
the object of an intensive research also in different ranges of the
interaction parameters, for their possible applications to a number
of other physical models, such as the study of incoherent solutions
in nonlinear optics.  We refer the reader to the recent papers
\cite{ac,mmp,dww,lw2,ww2} mainly dealing with systems of two
equations. For the general $k$--systems we refer to \cite{lw,si} and
the references therein.




\section{Preliminaries and main results}\label{sec:basic}
In the absence of a magnetic trapping potential we shall work in the
Sobolev space of radial functions $\spc$, endowed with the standard
norm $\|U\|^2=\int_{\R^N}|\nabla U_i|^2+U_i^2\,dx$; it is well known
that such functions are continuous everywhere but the origin, thus
we are allowed to evaluate them pointwise. As $N=2,3$ implies that
$p=4$ is a subcritical exponent, the (compact) embedding of $\spc$
in $L^4(\R^N)$ (see \cite{s}) will be available:
\begin{lemma}[(Sobolev--Strauss)]\label{lem:strauss}
If $U\in\spc$ then $\int_{\R^N}U^4\,dx\leq C_S^4 \|U\|^4$, and the
immersion $\spc\hookrightarrow L^4(\R^N)$ is compact for $N=2,3$.
\end{lemma}

 We search for solutions of
\eqref{eq:sys} as critical points of the related energy functional
\[
J_\b(U_1,\dots,U_k)=\sum_{i=1}^k\left[\frac12\|U_i\|^2-\frac14\int_{\R^N}
U_i^4\,dx\right]+\frac\b4\sum_{{i,j=1 \atop i\neq
j}}^k\int_{\R^N}U_i^2U_j^2\,dx
\]
(we will always omit the dependence on $\b$ when no confusion
arises). In the same way we associate with equation \eqref{eq:sing} the corresponding functional
\[
J^*(W)=\frac12\|W\|^2-\frac14\int_{\R^N} W^4\,dx.
\]
Let $h\in\N$ be fixed. We introduce the set of the nodal components
of radial functions having (at least) $h-1$ ordered zeroes as
\[
\Xh^*=\left\{(w_1,\dots w_h)\in\left(\spc\right)^h:\,
 \begin{array}{cl}
  \text{for every }l=1,\dots,h\text{ it holds } w_l\geq0,\,w_l\not\equiv0 \text{ and}\smallskip\\
  w_l(x_0)>0 \implies w_p(x)=0 \left\{
  \begin{array}{l}
   \forall|x|\geq|x_0|\text{ if }p<l\smallskip\\
   \forall|x|\leq|x_0|\text{ if }p>l.
  \end{array}
  \right.
\end{array}
\right\}.
\]
We will often write $W=(w_1,\dots w_h)$. By definition, if
$W\in\Xh^*$, then $w_l\cdot w_p=0$ a.e. when $l\neq p$. More
precisely, the sets $\{w_l>0\}$ are contained in disjoint
annuli\footnote{Here and in the following by annuli we mean also
balls or exteriors of balls.} and, for $l<p$, the annulus containing
$\{w_l>0\}$ is closer to the origin than the one containing
$\{w_p>0\}$. As a consequence, we have
$J^*(\sum_lw_l)=J^*(\sum_l(-1)^lw_l)=\sum_lJ^*(w_l)$.

We are interested in solutions of \eqref{eq:sing} with $h$ nodal
regions. The \emph{Nehari manifold} related to this problem is
defined as
\[
\Neh^*=\left\{W\in \Xh^*:\,J^*(w_l)=\sup_{\l>0}J^*(\l
w_l)\right\}=\left\{W\in \Xh^*:\,\|w_l\|^2=\int_{\R^N}
w_l^4\,dx\right\}.
\]
As a matter of fact one has
\begin{prop}\label{prop:sing}
Let
\[
c_\infty=\inf_{W\in\Neh^*}J^*(W)=\inf_{W\in\Xh^*}\sup_{\l_l>0}J^*\left(\sum_{l=1}^h\l_lw_l\right).
\]
Then the set
\[
\Kah=\left\{W\in\Neh^*:\,J^*(W)=c_\infty\right\}
\]
is non empty and compact, and, for every $W\in\Kah$, the functions
\[
\pm\sum_{l=1}^h(-1)^hw_l\text{ solve \eqref{eq:sing}.}\footnote{As a
consequence $\supp w_l$ is an annulus for every $l$, and $\supp
W=\R^N$.}
\]
Moreover there exist two constants $0<C_1<C_2$ such that, for
every $W\in\Kah$ and for every $l$ it holds
\[
C^2_1\leq\|w_l\|^2=\int_{\R^N}w_l^4\,dx\leq C^2_2.
\]
\end{prop}
For the proof of this result, very well known in the literature, we
refer to \cite{bw}.

Now, let us consider system \eqref{eq:sys}. Roughly speaking, we
want to construct solutions of \eqref{eq:sys} in the following way:
each $U_i>0$ is the sum of pulses $u_{im}$, where each $u_{im}$ is
near some $w_l$ for an appropriate $W\in\Kah$. Maybe an example will
make the situation more clear.

\begin{example}\label{example}
Let $h=5$ and $k=3$. A possible setting is to search for solutions
$U_1=u_{11}+u_{12}$, $U_2=u_{21}+u_{22}$, $U_3=u_{31}$, in such a
way that, for some $W\in\Kah$, (for instance) $u_{11}$ is near
$w_1$, $u_{21}$ is near $w_2$, $u_{12}$ is near $w_3$, $u_{31}$ is
near $w_4$, and $u_{22}$ is near $w_5$. The only rule we want to
respect is that two consecutive pulses $w_l$ and $w_{l+1}$ must
belong to different components $U_i$ and $U_j$ (see the last figure
in the introduction).
\end{example}

The general situation can be treated as follows. Let $h\geq k$ and
consider any surjective map
\[
\s:\{1,\dots,h\}\to\{1,\dots,k\}\quad\text{ such that
}\quad\s(l+1)\neq\s(l)\text{ for }l=1,\dots,h-1
\]
(a map that associates each pulse of an element of $\Kah$ to a
component $U_i$). The numbers $h_i=\#\s^{-1}(i)$ (the number of
pulses associated to the $i$--th component) are such that $h_i\geq1$
and $\sum_1^kh_i=h$. This means that we can (uniquely) define a
bijective map onto the set of double indexes
\[
\tilde\s:\{1,\dots,h\}\to\bigcup_{i=1}^k\{(i,m):\,m=1,\dots,h_i\}
\]
where the first index of $\tilde\s$ is given by $\s$, and the second
is increasing (when the first is fixed). In this setting, Example
\ref{example} can be read as $\tilde\s(1)=(1,1)$,
$\tilde\s(2)=(2,1)$, $\tilde\s(3)=(1,2)$, $\tilde\s(4)=(3,1)$,
$\tilde\s(5)=(2,2)$.\footnote{For easier notation we will write
$u_{im}$ instead of $u_{(i,m)}$; from now on we will use the letters
$i$, $j$ for the first index and the letters $m$, $n$ for the
second, while $l$ is reserved for the components of $W$.}.

According to the previous notation we define, for $\eps\leq1$,
\[
\Xh_\ep=\left\{(u_{11},\dots,u_{kh_k})\in\left(\spc\right)^h:\,\begin{array}{c}u_{im}\geq0
\text{ and there exists }W\in\Kah\text{ such that}\\
\displaystyle\sum_{i=1}^k \sum_{m=1}^{h_i}
\left\|u_{im}-w_{\tilde\s^{-1}(im)}\right\|^2<\ep^2\end{array}\right\},
\]
and $U_i=\sum_{m=1}^{h_i}u_{im}$. Sometimes we will use the distance
\[
d^2_{\tilde
\sigma}\left((u_{11},\dots,u_{kh_k}),W\right)=\sum_{i=1}^k
\sum_{m=1}^{h_i}\left\|u_{im}-w_{\tilde\s^{-1}(im)}\right\|^2
\]
\begin{rem}\label{lem:base}
Using Proposition \ref{prop:sing} it is easy to see that
\begin{enumerate}
 \item $\Xh_\ep$ is contained in an $\ep$--neighborhood of $\Kah$, in the
 sense of $d_{\tilde\sigma}$; $\Kah\subset\Xh_\ep$ (understanding the identification $w_l=u_{\tilde\s(im)}$);
 \item there exist constants $C_1$, $C_2$ not depending on $\b$ and
 $\ep<1$, such that $0<C_1\leq\|u_{im}\|\leq C_2$, $0<C_1^2\leq\int_{\R^N}u^4_{im}\,dx\leq
 C_2^2$
 \item $m\neq n$ implies $\int_{\R^N}\nabla u_{im} \cdot \nabla u_{in}<C\ep$,
 $\int_{\R^N}u_{im}u_{in}<C\ep$,
 $\int_{\R^N}U_i^2u_{im}u_{in}<C\ep$.
\end{enumerate}
\end{rem}
A first important result we want to give, that is underlying the
spirit of this whole paper, is the following: in the classical
Nehari's method described above, it is not necessary to perform the
min--max procedure on pulses with disjoined support, but we can
``mix up'', even tough not too much, the pulses with non adjacent
supports.
\begin{prop}\label{prop:nehari_mixed}
There exist $\ep_0\leq1$ such that for every $0<\ep<\ep_0$ the
following hold. If $(v_{11},\dots,v_{kh_k})\in\Xh_\ep$ is such that
\[
V_i\cdot V_j=0\quad\text{almost everywhere, for every }i,j,
\]
(but $v_{im}\cdot v_{in}$ is not necessarily null) then
\[
\sup_{\l_{im}>0}J^*\left(\sum_{i,m}\l_{im}v_{im}\right)\geq
c_\infty.
\]
\end{prop}
\begin{proof}
To prove the result, we will construct a $h$--tuple $(\tilde
w_1,\dots,\tilde w_h)\in\Neh^*$ such that, for a suitable choice of
the positive numbers $\tilde\lambda_{im}$'s, it holds
\[
J^*\left(\sum\tilde\l_{im}v_{im}\right)=J^*\left(\sum\tilde
w_{\tilde\sigma^{-1}(im)}\right).
\]
As a first step, we have to select the (disjointed) supports of such
$\tilde w_l$. To do that, using the definition of $\Xh_\ep$, we
choose a $W\in\Kah$ such that $\sum_{i,m}
\|v_{im}-w_{\tilde\s^{-1}(im)}\|^2<\ep^2$. When $\ep$ is small we
can find $h$ positive radii $r_1,\dots,r_h$ such that, for
$|x_l|=r_l$, $v_{\tilde\s(l)}(x_l)w_l(x_l)>0$.\footnote{Otherwise we
would obtain, for some $l$, $v_{\tilde\s(l)}w_l\equiv0$, and hence
$\ep^2>\|v_{\tilde\s(l)}-w_l\|^2=\|v_{\tilde\s(l)}\|^2+\|w_l\|^2\geq
2C^2_1$ by Remark  \ref{lem:base}.} Using this fact we can construct
the open connected annuli $A_1,\dots,A_h$ in such a way that
\begin{itemize}
 \item $|x_l|=r_l\text{ implies }x_l\in A_l$,
 \item $\bigcup_{l=1}^h\overline{A_l}=\R^N$ and $A_l\cap A_p=\emptyset\text{
       when }p\neq l$,
 \item
 $\mathrm{supp}V_i\subset\bigcup_{m=1}^{h_i}\overline{A_{\tilde\s^{-1}(im)}}$
       for every $i$
\end{itemize}
(recall that, by assumption, $\mathrm{int}(\mathrm{supp} V_i) \cap
\mathrm{int}(\mathrm{supp} V_j)=\emptyset$). By construction, we
obtain that obviously $\supp w_l\cap I_l\neq\emptyset$, while, by
connectedness,
\[
\supp w_{\tilde\s^{-1}(im)}\cap A_{\tilde\s^{-1}(in)}=\emptyset
\text{ when }n\neq m.
\]
In particular this last fact implies that, for $n\neq
m$,\footnote{Since $v_{im}$ vanishes on $\partial I_l$ for every
choice of the indexes, $v_{im}|_{I_l}$ belongs to $\spc$ and then,
when $\tilde\s(l)\neq(i,m)$,
$\ep^2>\|(v_{im}-w_{\tilde\s^{-1}(im)})|_{I_l}\|^2=\|v_{im}|_{I_l}\|^2$.}
\begin{equation}\label{eq:norme_ristrette}
\left\|\left.v_{im}\right|_{A_{\tilde\s^{-1}(in)}}\right\|^2\leq\ep^2\text{
and hence
}\left\|\left.v_{im}\right|_{A_{\tilde\s^{-1}(im)}}\right\|^2\geq
C_1^2-(h_i-1)\ep^2
\end{equation}
(with $C_1$ as in Remark \ref{lem:base}).  Now, depending on the
positive parameters $\l_{im}$'s, let us define the functions $\tilde
v_{im}$'s as
\[
\tilde v_{im}=\left.\left(\sum_{n=1}^{h_i}
\lambda_{in}v_{in}\right)\right|_{A_{\tilde\s^{-1}(im)}}.
\]
By construction we have that $\tilde v_{im}\cdot\tilde
v_{jn}\equiv0$ for every choice of the $\l_{im}$'s. We claim the
existence of $\tilde\l_{im}$'s such that the corresponding $\tilde
v_{im}$'s satisfy
\[
F_{im}(\tilde\l_{11},\dots,\tilde\l_{kh_k})=\|\tilde
v_{im}\|^2-\int_{\R^N}\tilde v_{im}^4\,dx=0\quad\text{ for every
}(i,m).
\]
This will imply that, writing $\tilde w_{l}=\tilde v_{\tilde\s(l)}$,
the $h$--tuple $(w_1,\dots,w_h)$ belongs to $\Neh^*$. Since
$J^*\left(\sum_{i,m}\tilde v_{im}\right) = J^*\left(\sum_{i,m}\tilde
\l_{im} v_{im}\right)$, this will conclude the proof of the lemma.
In order to prove the claim we will use ($k$ times) a classic result
by Miranda concerning the zeroes of maps from a rectangle of
$\R^{h_i}$ into $\R^{h_i}$ (see \cite{mir}), proving the existence,
when $\ep$ is sufficiently small, of constants $0<t\leq T$ such
that, for every $(i,m)$,
\begin{equation}\label{eq:miranda}
\l_{im}=T,\,t\leq\l_{in}\leq T \implies F_{im}<0,\qquad
\l_{im}=t,\,t\leq\l_{in}\leq T \implies F_{im}>0;
\end{equation}
from this and Miranda's Theorem the claim will follow. Let then
$(i,m)$ be fixed, $\l_{im}=T$ and, for $n\neq m$, $0\leq \l_{in}\leq
T$. Exploiting Remark  \ref{lem:base}, equation
\eqref{eq:norme_ristrette}, and the Sobolev embedding of $\spc$ in
$L^4(\R^N)$, we obtain\footnote{For easier notation we write
$l=\tilde\s^{-1}(im)$.}
\[
\left\|\tilde
v_{im}\right\|=\lambda_{im}\left\|\left.v_{im}\right|_{A_l}+\sum_{n\neq
m}\frac{\l_{in}}{\l_{im}}\left.v_{in}\right|_{A_l}\right\|\leq
T\left(C_2+(h_i-1)\ep\right)
\]
and
\[
\int_{\R^N}\tilde v_{im}^4\,dx\geq \l_{im}^4\int_{A_l} v_{im}^4\,dx=
\l_{im}^4\left(\int_{\R^N} v_{im}^4\,dx-\sum_{p\neq l}\int_{A_p}
v_{im}^4\,dx\right)\geq T^4\left(C_1-(h_i-1)C^4_S\ep^4\right).
\]
Choosing $\ep_0$ in such a way that, for $\ep<\ep_0$, the last term
is positive, we obtain an inequality of the form
\[
F_{im}\leq a T^2-bT^4<0\quad\text{if }T\text{ is fixed sufficiently
large,}
\]
and the first part of \eqref{eq:miranda} is proved. On the other
hand, let now $T$ be fixed as above, $\l_{im}=t$, and, for $n\neq
m$, $t\leq \l_{in}\leq T$. Using again the Sobolev embedding, we
have
\[
F_{im}=\|\tilde v_{im}\|^2-\int_{\R^N}\tilde v_{im}^4\,dx\geq
\|\tilde v_{im}\|^2 - C_S^4\|\tilde v_{im}\|^4=\|\tilde
v_{im}\|^2\left(1-C_S^4\|\tilde v_{im}\|^2\right).
\]
Then we simply have to prove that, when $t$ is sufficiently small,
$\|\tilde v_{im}\|<1/C^2_S$. As before, by Remark  \ref{lem:base}
and equation \eqref{eq:norme_ristrette}, we obtain
\[
\left\|\tilde
v_{im}\right\|=\left\|\left.\lambda_{im}v_{im}\right|_{A_l}+\sum_{n\neq
m}\left.\l_{in}v_{in}\right|_{A_l}\right\|\leq tC_2+(h_i-1)\ep T.
\]
Hence we can choose $t$ and $\ep_0$ sufficiently small in such a way
that, for every $\ep<\ep_0$, $\|\tilde v_{im}\|<1/C^2_S$. As we just
observed, this implies the second part of \eqref{eq:miranda},
concluding the proof of the proposition.
\end{proof}

\begin{coro}
Since by Remark  \ref{lem:base} we have $\Neh^*\subset\Xh_\ep$ for
every $\ep$, using the previous proposition we obtain the following
equivalent characterizations of $c_\infty$:
\[
\begin{split}
 c_\infty
 &= \inf_{\begin{array}{c}
            w_l\not\equiv0\\
            w_l\cdot w_p\equiv0
          \end{array}}
  \sup_{\l_l>0}J^*\left(\sum_l\l_lw_l\right)\smallskip\\
 &= \inf_{\begin{array}{c}
            (v_{im})\in\Xh_\ep\\
            V_i\cdot V_j\equiv0
          \end{array}}
  \sup_{\l_{im}>0}J^*\left(\sum_{i,m}\l_{im}v_{im}\right)\smallskip\\
 &= \inf_{\begin{array}{c}
            (v_{im})\in\Xh_\ep\\
            V_i\cdot V_j\equiv0
          \end{array}}
  \sup_{\l_{im}>0}J_\b\left(\sum_{m}\l_{1m}v_{1m},\dots,\sum_{m}\l_{km}v_{km}\right).
\end{split}
\]
\end{coro}
In the same spirit of the previous corollary, for $i=1,\dots,k$ and
$m=1,\dots,h_i$, let $\l_{im}\geq0$ and
$u_{im}\in\spc\setminus\{0\}$. We write
\[
\begin{array}{l}
 \L_i=(\l_{i1},\dots,\l_{ih_i}),\qquad U_i=\sum_m u_{im},\qquad \L_iU_i=\sum_m\l_{im}u_{im},\smallskip\\
 \Phi_\b(\L_1,\dots,\L_k)=J_\b(\L_1U_1,\dots,\L_kU_k),
\end{array}
\]
in such a way that $\Phi$ is a $C^2$--function. Moreover, let
\[
M_\b(u_{11},\dots,u_{kh_k})=\sup_{\l_{im}>0}\Phi_\b(\L_1,\dots,\L_k)
\]
and finally
\begin{equation}\label{eq:c_beta}
c_{\ep,\b}=\inf_{\Xh_\ep}M_\b.
\end{equation}
Our main results are the following.
\begin{teo}\label{teo:main}
There exist $\bar\ep>0$ and $\bar\b=\bar\b(\bar\ep)$ such that, if
$0<\ep<\bar\ep$ and $\b>\bar\b$, then $c_{\ep,\b}$ is a critical
value for $J_\b$, corresponding to a solution of \eqref{eq:sys}
belonging to $\Xh_\ep$.
\end{teo}

\begin{teo}\label{teo:main2}
Let $0<\ep<\bar\ep$ (as in the previous theorem) be fixed and
$(\b_s)_{s\in\N}$ be such that $\b_s\to+\infty$. Finally, let
$(U_1^s,\dots,U_k^s)$ be any solution of \eqref{eq:sys} at level
$c_{\ep,\b}$ and belonging to $\Xh_\ep$. Then, up to subsequences,
\[
d_{\tilde\sigma}\left((U_1^s,\dots,U_k^s),\Kah\right)\to0
\]
as $s\to+\infty$.
\end{teo}

\section{Estimates for any $\beta$ and $\ep$ small}

Let us start with some estimates on $c_{\ep,\b}$.
\begin{lemma}\label{lem:c_b-bdd}
When $\ep$ is fixed, $c_{\ep,\b}$ is non decreasing in $\b$, and
$c_{\ep,\b}\leq c_\infty$.
\end{lemma}
\begin{proof}
First of all, if $\b_1<\b_2$, then for any
$(u_{11},\dots,u_{kh_k})\in\Xh_\ep$ we have
$J_{\b_1}(\L_1U_1,\dots,\L_kU_k)< J_{\b_2}(\L_1U_1,\dots,\L_kU_k)$,
and we can pass to the inf--sup obtaining $c_{\ep,\b_1}\leq
c_{\ep,\b_2}$. Now, let $\b$ be fixed. By definition we have
\[
c_\infty=\inf_{\Xh_0}M_\b=\inf\left\{M_\b(u_{11},\dots,u_{kh_k}):\,\text{there
exists } W\in\Xh^*\text{ such that
}u_{im}=w_{\tilde\s^{-1}(im)}\right\}.
\]
Indeed, in such situation $w_lw_p=0$ for $l\neq p$ and thus, letting
$\mu_l=\l_{\tilde\s(l)}$, we obtain
$J_\b(\L_1U_1,\dots,\L_kU_k)=J^*(\sum\mu_lw_l)$. To conclude we
simply observe that
\[
\Xh_0\subset\Xh_\ep\qquad\text{ for every }\ep.\qedhere
\]
\end{proof}

\begin{coro}\label{coro:tilde_X}
Let
\[
\tilde\Xh_\ep=\left\{(u_{11},\dots,u_{kh_k})\in\Xh_\ep:\,M_\b(u_{11},\dots,u_{kh_k})<
c_\infty+\min\left(1,\frac{1}{\b}\right)\right\}.
\]
Then
\[
c_{\ep,\b}=\inf_{\tilde\Xh_\ep}M_\b.
\]
\end{coro}
From now on, we will restrict our attention to the elements of
$\tilde\Xh_\ep$. We remark that $\tilde\Xh_\ep$ depends on $\b$:
actually, if, for some $i\neq j$, $U_iU_j\neq0$ on a set of positive
measure, then the corresponding $(u_{11},\dots,u_{kh_k})$ may not
belong to $\tilde\Xh_\ep$ if $\b$ is sufficiently large.
Nevertheless, all the results we will prove in this section will
depend only on $\ep$, and not on $\b$.
\begin{lemma}\label{lem:achieve_M}
Let $(u_{11},\dots,u_{kh_k})\in\tilde\Xh_\ep$. Then
$M_\b(u_{11},\dots,u_{kh_k})$ is positive and achieved:
\[
0<M_\b(u_{11},\dots,u_{kh_k})=\Phi_\b(\bar\L_1,\dots,\bar\L_k),
\]
with
\[
\nabla\Phi_\b(\bar\L_1,\dots,\bar\L_k)\cdot(\bar\L_1,\dots,\bar\L_k)=0.
\]
\end{lemma}
\begin{proof}
We drop the dependence on $\b$. We observe that $\Phi$ is the sum of
two polynomials, which are homogenous of degree two and four,
respectively:
\[
\Phi(\L_1,\dots,\L_k)=\frac12P_2(\L_1,\dots,\L_k)+\frac14P_4(\L_1,\dots,\L_k)
\]
where $P_2=\sum_{i,m}\|\l_{im}u_{im}\|^2$. Therefore, for $t\geq0$,
\[
\Phi(t\L_1,\dots,t\L_k)=\frac12t^2P_2(\L_1,\dots,\L_k)+\frac14t^4P_4(\L_1,\dots,\L_k),
\]
Since $u_{im}\not\equiv0$ for every $i$, $m$, if some $\l_{im}$ is
different from 0 then $P_2>0$. Thus, for $t$ small,
$0<\Phi(t\L_1,\dots,t\L_k)\leq M$.\\ As a consequence, we can write
\[
M=\sup\left\{\Phi(t\L_1,\dots,t\L_k):\,t>0,\,\sum_{i=1}^k|\L_i|^2=1\right\}.
\]
On one hand, we have
\[
\max_{\sum|\L_i|=1}P_2(\L_1,\dots,\L_k)=a>0.
\]
On the other hand, since $M<c_\infty+1/\b$, then
\[
\max_{\sum|\L_i|=1}P_4(\L_1,\dots,\L_k)=-b<0
\]
(otherwise we would have $M=+\infty$). But then
\[
\Phi(t\L_1,\dots,t\L_k)\leq\frac{a}{2}t^2-\frac{b}{4}t^4<0\qquad\text{when
} t^2>\frac{2a}{b},
\]
thus $\Phi$ is negative outside a compact set, and $M$ is achieved
by some $(\bar\L_1,\dots,\bar\L_k)$.\\ Finally, since this is a
maximum for $\l_{im}\geq0$, $\bar\l_{im}>0$ implies
$\partial_{\l_{im}}\Phi(\bar\L_1,\dots,\bar\L_k)=0$, therefore
\[
\partial_{\l_{im}}\Phi(\bar\L_1,\dots,\bar\L_k)\cdot\bar\l_{im}=0\quad\text{ for every
}(i,m).
\qedhere
\]
\end{proof}
Thus, if $(u_{11},\dots,u_{kh_k})\in\tilde\Xh_\ep$, then $M_\b$ is a
maximum. We want to prove that, when $\ep$ is small (not depending
on $\b$), the maximum point is uniquely defined, and it smoothly
depends on $(u_{11},\dots,u_{kh_k})$. As a first step, we provide
some uniform estimates for its coordinates.

\begin{lemma}\label{lem:R}
There exists $R>0$, not depending on\footnote{Here $\ep_0$ is as in
Proposition \ref{prop:nehari_mixed}.} $\ep\leq\ep_0$ and $\b$, such
that, for every $(u_{11},\dots,u_{kh_k})\in\tilde\Xh_\ep$,
\begin{enumerate}
\item
$\nabla\Phi(\bar\L_1,\dots,\bar\L_k)\cdot(\bar\L_1,\dots,\bar\L_k)=0$
implies $\sum_i|\bar\L_i|^2< R^2$;
\item
$\sum_i|\L_i|^2= R^2$ implies
$\nabla\Phi(\L_1,\dots,\L_k)\cdot(\L_1,\dots,\L_k)<0$.
\end{enumerate}
\end{lemma}
\begin{proof}
Since $\Phi\leq M$, using the notations of the proof of the previous
lemma we can write
\[
\Phi(\bar\L_1,\dots,\bar\L_k)=\frac12P_2(\bar\L_1,\dots,\bar\L_k)+
\frac14P_4(\bar\L_1,\dots,\bar\L_k)< c_\infty+1
\]
(recall Corollary \ref{coro:tilde_X}) and
\[
\nabla\Phi(\bar\L_1,\dots,\bar\L_k)\cdot(\bar\L_1,\dots,\bar\L_k)=
P_2(\bar\L_1,\dots,\bar\L_k)+P_4(\bar\L_1,\dots,\bar\L_k)=0,
\]
providing
\begin{equation}\label{eq:bdd_norm}
P_2(\bar\L_1,\dots,\bar\L_k)=\sum_{i=1}^k\|\bar\L_iU_i\|^2<4(c_\infty+1).
\end{equation}
Since every $u_{im}$ is non negative, we have
$\|\bar\L_iU_i\|^2\geq\sum_m\bar\l_{im}^2\|u_{im}\|^2$. But we know
(Remark  \ref{lem:base}) that each $\|u_{im}\|$ is bounded from
below, providing
\[
\sum_{i=1}^k|\bar\L_i|^2<\frac{4(c_\infty+1)}{C_1^2}=R^2.
\]
Now let $(\L_1,\dots,\L_k)$ be fixed with $\sum_i|\L_i|^2= R^2$. For
$t>0$ we write
\[
f(t)=\nabla\Phi(t\L_1,\dots,t\L_k)\cdot(t\L_1,\dots,t\L_k)=
t^2P_2(\L_1,\dots,\L_k)+t^4P_4(\L_1,\dots,\L_k),
\]
and we know, from the discussion above, that $f(\bar t)=0$ implies
$\bar t<1$. Recalling that $P_4$ must be negative (otherwise
$M=+\infty$) we deduce that $f(1)<0$, concluding the proof.
\end{proof}

\begin{rem}\label{rem:bdd_three_addends}
As a consequence of the previous proof (and of Lemma
\ref{lem:achieve_M}) we have that, if
$(u_{11},\dots,u_{kh_k})\in\tilde\Xh_\ep$ and
$M_\b(u_{11},\dots,u_{kh_k})=\Phi_\b(\bar\L_1,\dots,\bar\L_k)$, then
the three quantities
\[
\sum_{i=1}^k\|\bar\L_iU_i\|^2,\qquad \sum_{i=1}^k\int_{\R^N}
\left(\bar\L_iU_i\right)^4\,dx,\qquad\b\sum_{{i,j=1 \atop i\neq
j}}^k\int_{\R^N}\left(\bar\L_iU_i\right)^2\left(\bar\L_jU_j\right)^2\,dx,
\]
are bounded not depending on $\b$: the first by equation
\eqref{eq:bdd_norm}; the second by the first bound and by the
continuous immersion of $\spc$ in $L^4(\R^N)$; the third by the
previous bounds and the fact that $M_\b< c_\infty+1$.
\end{rem}

\begin{lemma}\label{lem:lambda>1/2}
There exists $0<\ep_1\leq\ep_0$ (not depending on $\beta$) such that
if $\ep<\ep_1$, $(u_{11},\dots,u_{kh_k})\in\tilde\Xh_\ep$, and
$\Phi_\b(\bar\L_1,\dots,\bar\L_k)=M_\b(u_{11},\dots,u_{kh_k})$ then
\[
\bar\l_{im}>\frac12\qquad\text{for every }(i,m).
\]
In particular, since $\bar\l_{im}>0$,
$\nabla\Phi_\b(\bar\L_1,\dots,\bar\L_k)=0$.
\end{lemma}
\begin{proof}
By the previous lemma we know that $0\leq\bar\l_{im}\leq R$. We
choose $(i,m)$ and, for any $(j,n)\neq(i,m)$, we fix
$0\leq\l_{jn}\leq R$. We will prove that (for $\ep$ sufficiently
small)
\begin{equation}\label{eq:convex}
\Phi|_{\l_{im}=0}<\Phi|_{\l_{im}=1/2}, \qquad\text{ and }\qquad
0\leq\l_{im}\leq1/2 \implies
\partial^2_{\l^2_{im}}\Phi(\L_1,\dots\L_k)>0,
\end{equation}
and the result will follow. Let $V=\sum_{n\neq m}\l_{in}u_{in}$.
Since all the $\l_{in}$'s are bounded, by Remark  \ref{lem:base} we
know that
\[
\left\langle V,u_{im} \right\rangle=o(1),\quad\int_{\R^N}V^p
u_{im}^{q}\,dx=o(1),\qquad\text{ as }\ep\to0.
\]
Moreover, according to the definition of $\Xh_\ep$, let
$l=\tilde\s^{-1}(im)$ and $w_l$ be such that $\|u_{im}-w_l\|<\ep$,
with $\|w_l\|^2=\int_{\R^N}w_l^4\,dx$. On one hand, we have (as
$\ep\to0$)
\[
\begin{split}
\Phi|_{\l_{im}=1/2}-\Phi|_{\l_{im}=0}&\geq\frac12\left[\left\|V+\frac12u_{im}
  \right\|^2-\left\|V\right\|^2\right] - \frac14\int_{\R^N}\left[\left(V+
  \frac12u_{im}\right)^4-\left(V\right)^4\right]\,dx\\
 &= \frac12\left\|\frac12u_{im}\right\|^2-\frac14\int_{\R^N}\left(
  \frac12u_{im}\right)^4\,dx+o(1)\\
 &= \frac12\left\|\frac12w_l\right\|^2-\frac14\int_{\R^N}\left(
  \frac12w_l\right)^4\,dx+o(1)\\
 &=\frac{7}{64}\|w_l\|^2+o(1)>0.
\end{split}
\]
On the other hand, with similar calculations, we obtain
\begin{equation}\label{eq:grad_Phi}
\partial_{\l_{im}}\Phi = \left\langle\L_{i}U_i,u_{im}\right\rangle -
\int_{\R^N}(\L_{i}U_i)^3u_{im}\,dx+\b\int_{\R^N}(\L_iU_i)u_{im}\sum_{j\neq
i}(\L_jU_j)^2\,dx
\end{equation}
and
\begin{equation}\label{eq:der_sec_Phi}
\begin{split}
\partial^2_{\l^2_{im}}\Phi(\L_1,\dots\L_k)&=\|u_{im}\|^2 -
  3\int_{\R^N}(\L_{i}U_i)^2u_{im}^2\,dx+\b\int_{\R^N}u_{im}^2\sum_{j\neq
  i}(\L_jU_j)^2\,dx\\
 &\geq \|w_l\|^2 - 3\int_{\R^N}\l_{im}^2w_l^2\,dx+o(1)\\
 &= (1-3\l_{im}^2)\|w_l\|^2+o(1)>0
\end{split}
\end{equation}
since $\l_{im}\leq1/2$.
\end{proof}
\begin{rem}\label{rem:1/2}
As a byproduct of the previous proof (equation \eqref{eq:convex}),
we have that, if $\ep<\ep_1$, $(u_{11},\dots,u_{kh_k})\in
\tilde\Xh_\ep$, and $0\leq\l_{jn}\leq R$, then
\[
\partial_{\l_{im}}\Phi(\L_1,\dots\L_k)|_{\l_{im}=1/2}>0.
\]
\end{rem}

\begin{lemma}\label{lem:pos_def_hess}
There exists $0<\eps_2\leq\eps_1$ (not depending on $\b$) such that
if $\ep<\ep_2$, $(u_{11},\dots,u_{kh_k})\in\tilde\Xh_\ep$ and
$\nabla\Phi(\bar\L_1,\dots,\bar\L_k)=0$ with $\bar\l_{im}>1/2$ then
the Hessian matrix
\[
D^2\Phi(\bar\L_1,\dots,\bar\L_k)\text{ is negative definite.}
\]
\end{lemma}

\begin{proof}
From \eqref{eq:grad_Phi} we obtain
\[
\begin{split}
 \partial^2_{\l_{im}\l_{jn}}\Phi &= \b\int_{\R^N}2(\L_iU_i) (\L_jU_j) u_{im} u_{jn}\,dx \qquad
  \text{if }j\neq i\\
 \partial^2_{\l_{im}\l_{in}}\Phi &= \left\langle u_{im},u_{in}\right\rangle - 3\int_{\R^N}
 (\L_{i}U_i)^2u_{im}u_{in}\,dx+\b\int_{\R^N} u_{im}u_{in}\sum_{j\neq i}(\L_jU_j)^2\,dx
\end{split}
\]
(we computed $\partial^2_{\l^2_{im}}\Phi$ in
\eqref{eq:der_sec_Phi}). For $m\neq n$ we write
\[
M^i_{mn}=\left\langle u_{im},u_{in}\right\rangle
-\int_{\R^N}(\bar\L_iU_i)^2 u_{im}u_{in}\,dx
\]
in such a way that
\[
\left\langle\bar\L_i U_i,u_{im}\right\rangle
-\int_{\R^N}(\bar\L_iU_i)^3u_{im}\,dx= \bar\l_{im}\left[\left\|
u_{im}\right\|^2 - \int_{\R^N}(\bar\L_i U_i)^2 u_{im}^2\,dx
+\sum_{n\neq m}\frac{\bar\l_{in}}{\bar\l_{im}}M^i_{mn}\right].
\]
Since $\nabla\Phi(\bar\L_1,\dots,\bar\L_k)=0$ we have, for every
$i$, $m$,
\[
\int_{\R^N}(\bar\L_iU_i)^2u_{im}^2= \|u_{im}\|^2+ \sum_{n\neq
m}\frac{\bar\l_{in}}{\bar\l_{im}}M^i_{mn}+
\frac{\b}{\bar\l_{im}}\int_{\R^N}(\bar\L_iU_i)u_{im}\sum_{j\neq
 i}(\bar\L_jU_j)^2\,dx.
\]
Substituting we obtain
\begin{equation}\label{eq:hess1}
 \partial^2_{\l_{im}\l_{im}}\Phi(\bar\L_1,\dots,\bar\L_k)=
 -2\underbrace{\|u_{im}\|^2}_{(A)}
 -\underbrace{3\sum_{n\neq m}\frac{\bar\l_{in}}{\bar\l_{im}}M^i_{mn}}_{(B)}
 +\b\int_{\R^N}\underbrace{u_{im}\left(u_{im}-\frac{3}{\bar\l_{im}}\bar\L_iU_i\right)
 \sum_{j\neq i}(\bar\L_jU_j)^2}_{(C)}\,dx,
\end{equation}
and, for $n\neq m$,
\begin{equation}\label{eq:hess2}
 \partial^2_{\l_{im}\l_{in}}\Phi(\bar\L_1,\dots,\bar\L_k) =
 \underbrace{M^i_{mn}-2\int_{\R^N}(\bar\L_iU_i)^2u_{im}u_{in}\,dx}_{(B)}
 +\b\int_{\R^N}\underbrace{u_{im}u_{in}\sum_{j\neq i}(\bar\L_jU_j)^2}_{(C)}\,dx.
\end{equation}
As a consequence we can split $D^2\Phi$ as
\[
D^2\Phi(\bar\L_1,\dots,\bar\L_k)=-2A+B+\b\int_{\R_N} C(x)\,dx,
\]
where each of the matrices $A$, $B$, and $C$ contains the
corresponding terms in \eqref{eq:hess1} and \eqref{eq:hess2}, and
$C$ also contains the terms appearing in $\partial^2_{\l_{im}
\l_{jn}}\Phi$, $i\neq j$. First of all, using Remark  \ref{lem:base}
and the boundedness of the $\bar\l_{im}$'s, we observe that $A$ is
diagonal and strictly positive definite, independent of $\eps$,
while $B$ is arbitrary small as $\eps$ goes to zero (not depending
on $\b$). We will show that $C(x)$ is negative semidefinite for
every $x$: this will conclude the proof.

To do that, we will only use that $u_{im}\geq0$ and
$\sum_m\bar\l_{im}u_{im}=\bar\L_iU_i$, therefore, without loss of
generality, we can put $\bar\l_{im}=1$ for every
$(i,m)$\footnote{replacing $\bar\l_{im}u_{im}$ with $u_{im}$ and
$\bar\L_iU_i$ with $U_i$.}. The matrix $C(x)$ can be written as the
sum of matrices $C_{ij}(x)$ where only two components, say $U_i$ and
$U_j$, interact. Such matrices, for $x$ fixed, contain many null
blocks, corresponding both to the interaction with the other
components $U_p$, $p\neq i,j$, and to the pulses of $U_i$ and $U_j$
vanishing in $x$. All those null blocks do not incide on the
semidefiniteness of $C_{ij}$; up to the null terms, $C_{ij}$ writes
like
\[
\left(
  \begin{array}{cc}
     \begin{array}{ccc}
      U_j^2(u_{i1}-3U_i)u_{i1} & \cdots & U_j^2u_{i1}u_{ih_i} \\
      \vdots & \ddots & \vdots \\
      U_j^2u_{ih_i}u_{i1} & \cdots & U_j^2(u_{ih_i}-3U_i)u_{ih_i}
     \end{array}
    \vline&
     \begin{array}{ccc}
      2U_iU_ju_{i1}u_{j1} & \cdots & 2U_iU_ju_{i1}u_{jh_j} \\
      \vdots & \ddots & \vdots \\
      2U_iU_ju_{ih_i}u_{j1} & \cdots & 2U_iU_ju_{ih_i}u_{jh_j}
     \end{array}
    \\ \hline\hfill
     \begin{array}{ccc}
      2U_iU_ju_{i1}u_{j1} & \cdots & 2U_iU_ju_{ih_i}u_{j1} \\
      \vdots & \ddots & \vdots \\
      2U_iU_ju_{jh_j}u_{i1} & \cdots & 2U_iU_ju_{ih_i}u_{jh_j}
     \end{array}
    \hfill\vline&
     \begin{array}{ccc}
      U_i^2(u_{j1}-3U_j)u_{j1} & \cdots & U_i^2u_{j1}u_{jh_j} \\
      \vdots & \ddots & \vdots \\
      U_i^2u_{jh_j}u_{j1} & \cdots & U_i^2(u_{jh_j}-3U_j)u_{jh_j}
     \end{array}
  \end{array}
\right)
\]
(where every term is strictly positive), which has the same
signature than
\[
\left(
  \begin{array}{cc}
    \begin{array}{ccc}
      1-3U_i/u_{i1} & \cdots & 1 \\
      \vdots & \ddots & \vdots \\
      1 & \cdots & 1-3U_i/u_{ih_i}
    \end{array}
    \vline& 2\\
    \hline\smallskip
    \hfill 2
    \hfill\vline&
    \begin{array}{ccc}
      1-3U_j/u_{j1} & \cdots & 1 \\
      \vdots & \ddots & \vdots \\
      1 & \cdots & 1-3U_j/u_{jh_j}
    \end{array}
  \end{array}
\right)
\]
(we mean that in the two blocks every term is equal to 2). The last
matrix can be seen as the sum
\[ 2\left(
  \begin{array}{cc}
    -1\ \vline& 1\\
    \hline\smallskip
    \hfill 1\
    \hfill\vline&-1
  \end{array}
\right) + 3\left(
  \begin{array}{cc}
    \begin{array}{ccc}
      1-1/\a_1 & \cdots & 1 \\
      \vdots & \ddots & \vdots \\
      1 & \cdots & 1-1/\a_{h_i}
    \end{array}
    \vline& 0\\
    \hline\smallskip
    \hfill 0
    \hfill\vline&
    \begin{array}{ccc}
      1-1/\b_1 & \cdots & 1 \\
      \vdots & \ddots & \vdots \\
      1 & \cdots & 1-1/\b_{h_j}
    \end{array}
  \end{array}
\right),
\]
where $\a_m=u_{im}/U_i$, $\b_m=u_{jm}/U_j$, in such a way that
$\sum\a_m=\sum\b_m=1$. It is easy to see that the first addend is
negative semidefinite, so the last thing we have to prove is that
$\sum_m\a_m=1$, $\a_m>0$, implies that
\[
D= \left(
\begin{array}{ccc}
      1-1/\a_1 & \cdots & 1 \\
      \vdots & \ddots & \vdots \\
      1 & \cdots & 1-1/\a_{h}
    \end{array}
\right)
\]
is negative semidefinite. Let $\xi=(\xi_1,\dots,\xi_h)\in\R^h$. Then
it is easy to prove that
\[
\sum_{m=1}^h\frac{\xi_m^2}{\a_m^2}\geq h|\xi|^2,\qquad \text{ thus
}\qquad D\xi\cdot\xi\leq\left(\sum_{m=1}^h\xi_m\right)^2-h|\xi|^2,
\]
that is trivially non positive for every $h$ and $\xi$.
\end{proof}

\begin{prop}\label{prop:unique_lambda}
Let $\ep<\ep_2$ in such a way that all the previous results hold.
Then, for every $(u_{11},\dots,u_{kh_k})\in\tilde\Xh_\ep$ there
exists one and only one choice
\[
\bar\L_i=\bar\L_i(u_{11},\dots,u_{kh_k})
\]
such that
\[
J_\b\left(\bar\L_1(u_{11},\dots,u_{kh_k})U_1,\dots,\bar\L_k(u_{11},
\dots,u_{kh_k})U_k\right)=M_\b(u_{11},\dots,u_{kh_k}).
\]
Moreover, each $\bar\L_i$ is well defined and of class $C^1$ on a
neighborhood $N(\tilde\Xh_\ep)$ of $\tilde\Xh_\ep$.
\end{prop}
\begin{proof}
To start with we will show, via a topological degree argument, that
$(\bar\L_1,\dots,\bar\L_k)$ is uniquely defined on $\tilde\Xh_\ep$.
Indeed, consider the set
\[
D=\left\{(\L_1,\dots,\L_k)\in\R^h:\,\sum_{i}|\L_i|^2<R^2,\,\l_{im}>1/2\right\}.
\]
By Lemma \ref{lem:R} and Remark \ref{rem:1/2} we know that
$\nabla\Phi$ points inward on $\partial D$. As a consequence
$-\nabla\Phi$ is homotopically equivalent to a translation of the
identity map, and
\[
\deg\left(\nabla\Phi,0,D\right)=(-1)^h.
\]
On the other hand, such degree must be equal to the sum of the local
degrees of all the critical points of $\Phi$ in $D$: since these
points are all non degenerate maxima (by Lemma
\ref{lem:pos_def_hess}), and they have local degree $(-1)^h$, we
conclude that there is only one critical point of $\Phi$ in $D$, and
it must be the global maximum point. Therefore the maps
$\bar\L_i(u_{11},\dots,u_{kh_k})$ are well defined in
$\tilde\Xh_\ep$. Moreover, they are implicitly defined by
\[
\nabla\Phi(\bar\L_1,\dots,\bar\L_k)=0,
\]
thus to conclude we can apply the Implicit Function Theorem in a
neighborhood of any point of $\tilde\Xh_\ep$: indeed, $\nabla\Phi$
is a $C^1$ map (both in the $\l$--variables and in the
$u$--variables); moreover, its differential with respect to the
$\l$--variables is invertible by Lemma \ref{lem:pos_def_hess} (it is
simply $D^2\Phi$).
\end{proof}

We observe that, even if $(u_{11},\dots,u_{kh_k})$ belongs to
$\tilde\Xh_\ep$, nevertheless this might not be true for the
corresponding $(\bar\l_{11}u_{11},\dots,\bar\l_{kh_k}u_{kh_k})$. At
this point we can only state a weaker property for those elements of
$\tilde\Xh_\ep$ with the corresponding $U_i$'s having disjoint
supports.

\begin{lemma}\label{lem:restrict_ep_false}
Let $\ep<\ep_2$ in such a way that all the previous results hold,
and $(v_{11},\dots,v_{kh_k})\in\tilde\Xh_\ep$ be such that
\[
V_i\cdot V_j=0\quad\text{almost everywhere, for every }i,j.
\]
Then there exists $\delta=\delta(\ep)$ (not depending on $\b$) such
that
$(\bar\l_{11}v_{11},\dots,\bar\l_{kh_k}v_{kh_k})\in\tilde\Xh_{\delta}$.\footnote{Here
and in the following
$\bar\l_{im}=\bar\l_{im}(v_{11},\dots,v_{kh_k})$, according to
Proposition \ref{prop:unique_lambda}.} Moreover, $\d$ goes to 0 as
$\ep$ does.
\end{lemma}
\begin{proof}
To start with, we observe that $v_{im}\geq0$ implies
$\bar\l_{im}v_{im}\geq0$, and that
$M_\b(u_{11},\dots,u_{kh_k})=M_\b(\bar\l_{11}u_{11},\dots,\bar\l_{kh_k}u_{kh_k})$.
As a consequence, if we prove that
\[
\sum_{i,m}\|v_{im}-w_{\tilde\s^{-1}(im)}\|^2<\ep^2\quad\implies\quad
\sum_{i,m}\|\bar\l_{im}v_{im}-w_{\tilde\s^{-1}(im)}\|^2<\delta^2,
\]
with $\d$ vanishing when $\ep$ does, we have finished. By assumption
we have that $\b\int(\bar\L_iV_i)v_{im}(\bar\L_jV_j)^2=0$ for every
choice of the indexes, so that the functions $\bar\l$ are implicitly
defined by
\[
\left\langle\sum_n\bar\l_{in}v_{in},v_{im}\right\rangle-\int_{\R^N}\left(\sum_n
\bar\l_{in}v_{in}\right)^3v_{im}\,dx=0.
\]
On the other hand, by the definition of $\Kah$, we know that
\[
\left\langle\sum_n\bar\l_{in}w_{\tilde\s^{-1}(in)},w_{\tilde\s^{-1}(im)}
\right\rangle-\int_{\R^N}\left(\sum_n\bar\l_{in}w_{\tilde\s^{-1}(in)}
\right)^3w_{\tilde\s^{-1}(im)}\,dx=0\,\iff\,\bar\l_{in}=1\text{ for
every }n.
\]
But then our claim directly follows from the Implicit Function
Theorem.
\end{proof}

\section{Estimates for $\ep$ fixed and $\b$ large}

From now on we choose $\bar\ep>0$ in such a way that, for every
$\ep<\bar\ep$, it holds
\[
\ep<\ep_2\quad\text{and}\quad\d<\ep_2
\]
with $\ep_2$ as in Proposition \ref{prop:unique_lambda} and
$\d=\d(\ep)$ as in Lemma \ref{lem:restrict_ep_false}. As we said,
$\bar\ep$ do not depend on $\b$. In the following $\ep$ and $\d$ are
considered fixed as above.

Since in the following we will let $\b$ move, we observe that, as we
already remarked, the set $\tilde\Xh_\ep=\tilde\Xh_{\ep,\b}$ also
depends on $\b$, since the functions inside satisfy $M_\b<
c_\infty+1/\b$.
\begin{rem}
If $\b_1\leq\b_2$, then for any $(u_{11},\dots,u_{kh_k})\in\Xh_\ep$
and any choice of the $\L_i$'s, it holds
$J_{\b_1}(\L_1U_1,\dots,\L_kU_k)\leq
J_{\b_2}(\L_1U_1,\dots,\L_kU_k)$. Passing to the supremum we obtain
\[
\b_1\leq\b_2 \implies
\tilde\Xh_{\ep,\b_2}\subset\tilde\Xh_{\ep,\b_1}.
\]
\end{rem}
In the following we will deal with sequence of $h$--tuples in
$\tilde\Xh_\ep$, with increasing $\b$. For this reason, we start
this section with a general result about some convergence property
for such sequences.

\begin{lemma}\label{lem:weak_closure_infty}
Let the sequence $\b_s\to+\infty$, $s\in\N$, and let us consider a
sequence of $h$--tuples
\[
(u^s_{11},\dots,u^s_{kh_k})\in \tilde\Xh_{\ep,\b_s}.
\]
Then, up to a subsequence, $u^s_{im}\to u^*_{im}$, strongly in
$\spc$, for every $(i,m)$. Moreover
\[
U^*_i\cdot U^*_j\equiv0\text{ for }i\neq j,\quad\text{ and
}(u^*_{11},\dots,u^*_{kh_k})\in \tilde\Xh_{\ep,\b}\text{ for every
}\b.
\]
Finally, writing $\bar\l^*_{im}=\bar\l_{im}(u^*_{11},\dots,
u^*_{kh_k})$, we have that $\bar\l_{im}^s\to\bar\l_{im}^*$ and
\[
J_{\b}(\L^*_1U_1^*,\dots,\L^*_kU_k^*)=c_\infty\quad\text{ for every
}\b.
\]
\end{lemma}
\begin{proof}
By assumption we can find a sequence
$(w^s_{1},\dots,w^s_{l})\in\Kah$ such that
$\sum_{i,m}\|u^s_{im}-w^s_{l}\|^2<\ep^2$, where
$l=\tilde\s^{-1}(im)$. Since $\Kah$ is compact, we have that, up to
a subsequence, $w^s_l\to w_l^*$ strongly, and
$u^s_{im}\rightharpoonup u_{im}^*$ weakly in $\spc$ (since they are
bounded, independently on $\b$ (see Lemma \ref{lem:R}), also each
$\bar\l_{im}^\b$ converges to some number). By the compact immersion
of $\spc$ in $L^4(\R^N)$, we deduce that $u^s_{im}\to u_{im}^*$
strongly in $L^4(\R^N)$. We know by Remark
\ref{rem:bdd_three_addends} that
\[
\b_s\sum_{{i,j=1 \atop i\neq
j}}^k\int_{\R^N}\left(\bar\L^s_iU^s_i\right)^2\left(\bar\L^s_jU^s_j\right)^2\,dx\leq
C
\]
not depending on $\b$. Using the strong $L^4$--convergence and Lemma
\ref{lem:lambda>1/2} we conclude that $U^*_i\cdot U^*_j\equiv0\text{
for }i\neq j$. To prove that
$(u_{11}^*,\dots,u_{kh_k}^*)\in\tilde\Xh_\ep$ we observe that:
\begin{itemize}
 \item $u_{im}^*\geq0$ by the strong $L^4$--convergence;
 \item $\sum_{i,m}\|u_{im}^*-w_{l}^*\|^2<\ep^2$ by weak lower
  semicontinuity of $\|\cdot\|$;
 \item  finally, for every choice of the $\l_{im}$'s, we have that
  $\liminf\|\L_iU_i^s\|\geq\|\L_iU_i^*\|$, $\lim\int(\L_iU_i^s)^4=
  \int(\L_iU_i^*)^4$ and
  \[
  \liminf
  \b_s\int_{\R^N}\left(\bar\L_iU^s_i\right)^2\left(\bar\L_jU^s_j\right)^2\,dx\geq0
  =
  \b\int_{\R^N}\left(\bar\L_iU^*_i\right)^2\left(\bar\L_jU^*_j\right)^2\,dx
  \quad\text{ for every }\b,
  \]
 providing, for every $\b$,
  $J_\b(\L_{i}U_{i}^*) \leq\liminf
  J_{\b_s}(\L_{i}U_{i}^s)< c_\infty+1/\b_s$, that implies
  \[
  M_{\b}(u_{11}^*,\dots,u_{kh_k}^*)\leq c_\infty<c_\infty+\min(1,1/\b).
  \]
\end{itemize}
Thus $(u_{11}^*,\dots,u_{kh_k}^*)\in\tilde\Xh_{\ep,\b}$ and we can
write $\bar\l_{im}^*=\bar\l_{im}(u_{11}^*,\dots,u_{kh_k}^*)$. Now,
since $U_i^*\cdot U_j^*\equiv0$, by Proposition
\ref{prop:nehari_mixed} we know that
\[
J_{\b}(\bar\L^*_{1}U_{1}^*,\dots,\bar\L^*_{k}U_{k}^*)=
\sup_{\l_{im}>0}J^*\left(\sum_{i,m}\l_{im}v_{im}\right)\geq
c_\infty.
\]
On the other hand, for what we said,
\[
c_\infty+1/\b_s\geq
J_{\b_s}(\bar\L^s_{1}U_{1}^s,\dots,\bar\L^s_{k}U_{k}^s)\geq
J_{\b_s}(\bar\L^*_{1}U_{1}^s,\dots,\bar\L^*_{k}U_{k}^s)\geq
J_\b(\bar\L^*_{1}U_{1}^*,\dots,\bar\L^*_{k}U_{k}^*)+o(1),
\]
where the second inequality is strict if and only if
$\bar\l_{im}^\b\not\to\bar\l_{im}^*$, and the third is strict if and
only if $u_{im}^\b\not\to u_{im}^*$. Comparing the last two
equations, we obtain that
\[
J(\bar\L^*_{1}U_{1}^*,\dots,\bar\L^*_{k}U_{k}^*)= c_\infty,\quad
\bar\l_{im}^\b\to\bar\l_{im}^*,\quad\text{ and }u_{im}^\b\to
u_{im}^*\text{ strongly,}
\]
concluding the proof.
\end{proof}

Now we want to show that, if $\b$ is sufficiently large, then the
result of Lemma \ref{lem:restrict_ep_false} holds on the whole
$\tilde\Xh_\ep$, without restrictions.

\begin{lemma}\label{lem:restrict_ep_true}
There exists $\b_1$ such that if $\beta>\b_1$ then
\[
(u_{11},\dots,u_{kh_k})\in\tilde\Xh_\ep\quad\implies\quad
(\bar\l_{11}u_{11},\dots,\bar\l_{kh_k}u_{kh_k})\in\tilde\Xh_\delta.
\]
\end{lemma}
\begin{proof}
As in Lemma \ref{lem:restrict_ep_false}, since
$M_\b(u_{11},\dots,u_{kh_k})=M_\b(\bar\l_{11}u_{11},\dots,\bar\l_{kh_k}u_{kh_k})$,
and $u_{im}\geq0$ implies $\bar\l_{im}u_{im}\geq0$, we only have to
prove that, when $\b$ is sufficiently large,
\[
\sum_{i,m}\|u_{im}-w_{l}\|^2<\ep^2\quad\implies\quad
\sum_{i,m}\|\bar\l_{im}u_{im}-w_{l}\|^2<\delta^2,
\]
where $l=\tilde\s^{-1}(im)$. By contradiction, let $\b_s\to+\infty$
and $(u^s_{11},\dots,u^s_{kh_k})\in\tilde\Xh_\ep$ be such that
$\sum_{i,m}\|u^s_{im}-w^s_{l}\|^2<\ep^2$, and
$\sum_{i,m}\|\bar\l^s_{im}u^s_{im}-w_{l}\|^2\geq\delta^2$ for any
$(w_1,\dots,w_h)\in\Kah$. Using Lemma \ref{lem:weak_closure_infty}
we have that $u^s_{im}\to u^*_{im}$, $\bar\l_{im}^s\to
\bar\l_{im}^*$, in such a way that
\[
(u^*_{11},\dots,u^*_{kh_k})\in\tilde\Xh_\ep,\quad U^*_i\cdot
U^*_j\equiv0,\quad\text{and
}\sum_{i,m}\|\bar\l^*_{im}u^*_{im}-w_{l}\|^2\geq\d,
\]
for every $(w_1,\dots,w_h)\in\Kah$. But this is in contradiction
with Lemma \ref{lem:restrict_ep_false}.
\end{proof}
By the previous lemma we have that, for every
$(u_{11},\dots,u_{kh_k})\in\tilde\Xh_\ep$, the corresponding maximum
point $(\bar\l_{11}u_{11},\dots,\bar\l_{kh_k}u_{kh_k})$ belongs to
$\tilde\Xh_\d$ and $\bar\l_{im}(\bar\l_{11}u_{11},\dots,
\bar\l_{kh_k}u_{kh_k})=1$ for every $(i,m)$. Motivated by this fact
we define
\[
\Neh_\b=\left\{(u_{11},\dots,u_{kh_k})\in\tilde\Xh_{\d,\b}:\,
\bar\l_{im}(u_{11},\dots,u_{kh_k})=1\text{ for every }(i,m)\right\},
\]
immediately obtaining that, on $\Neh_\b$, $M_\b\equiv J_\b$ and
\begin{equation}\label{eq:c_beta_con_Nehari}
c_{\ep,\b}\geq\inf_{\Neh_\b}J_\b(U_1,\dots,U_k).
\end{equation}
As a matter of fact, if $\b$ is sufficiently large, also the
opposite inequality holds.
\begin{lemma}\label{lem:excision}
There exists $\b_2\geq\b_1$ such that, if $\b>\b_2$ then
\[
\Neh_\b\subset\tilde\Xh_{\ep/2}.
\]
\end{lemma}
\begin{proof}
We argue again by contradiction. Let (up to a subsequence)
$\b_s\to+\infty$, $(u^s_{11},\dots,u^s_{kh_k}) \in \Neh_\b$ (and
hence $\bar\l_{im}^s=1$) be such that
\begin{equation}\label{eq:out_of_eps/2}
\sum_{i,m}\|u^s_{im}-w_{l}\|^2\geq\frac{\ep^2}{4}\quad\text{ for
every }W\in\Kah.
\end{equation}
Using Lemma \ref{lem:weak_closure_infty}, we have that $u^s_{im}\to
u_{im}^*$ strongly in $\spc$ and $\bar\l_{im}^*=1$, for every
$(i,m)$. As a consequence, defining $w^*_{l}=u^*_{\tilde\sigma(l)}$,
we obtain that $J^*(w^*_l)=\sup_\l J^*(\l w^*_l)$ and $J(\sum_l
w_l^*)=c_\infty$. Therefore $(w^*_1,\dots,w^*_h)\in\Kah$, and,
obviously $\sum_{i,m}\|u^*_{im}-w^*_{l}\|^2=0$. But this, using
strong convergence in \eqref{eq:out_of_eps/2}, provides a
contradiction.
\end{proof}
\begin{rem}\label{rem:equiv_char}
Taking into account \eqref{eq:c_beta_con_Nehari}, and the previous
lemma (beside the inclusion $\tilde\Xh_{\ep/2}\subset\tilde\Xh_\ep$)
we obtain
\[
c_{\ep,\b}=\inf_{\tilde\Xh_\ep}M_\b\geq\inf_{\Neh_\b}J_\b(U_1,\dots,U_k)\geq
\inf_{\tilde\Xh_{\ep/2}}M_\b\geq\inf_{\tilde\Xh_\ep}M_\b=c_{\ep,\b},
\]
obtaining three equivalent characterizations of $c_{\ep,\b}$.
\end{rem}

\section{Proof of the main results}
In order to prove our main results, we present an useful abstract
lemma.

\begin{lemma}\label{lem:variational_lemma}
Let $H$ be an Hilbert space, $d$ an integer, $I\in C^2(H^d;\R)$,
$\Xh\subset H^d$, and $N(\Xh)$ an open neighborhood of $\Xh$. Let us
assume that:
\begin{enumerate}
 \item $d$ functionals $\bar\l_i\in C^1(N(\Xh);\R)$, $i=1,\dots,d$, are
 uniquely defined, in such a way that $\bar\l_i>0$ for every $i$ and
 \[
 \sup_{\l_i>0}I(\l_1x_1,\dots,\l_dx_d)=I(\bar\l_1(x_1,\dots,x_d)x_1,\dots,\bar\l_d(x_1,\dots,x_d)x_d);
 \]
 \item $(\bar x_1,\dots,\bar x_d)\in\Xh$ is such that
 $\bar\l_i(\bar x_1,\dots,\bar x_d)=1$ for every $i$,
 \[
 \inf_{\Xh}\sup_{\l_i>0}I(\l_1x_1,\dots,\l_dx_d)=I(\bar x_1,\dots,\bar
 x_d),
 \]
 and the $d\times d$ matrix $H=\left(\partial^2_{x_ix_j}I(\bar x_1,\dots,\bar x_d)[\bar x_i,\bar x_j]
 \right)_{i,j=1,\dots,d}$ is invertible;
 \item $(y_1,\dots,y_d)\in H^d$ is such that, for some $\bar t>0$ and $0<\d<1$,
 \[
 (s_1\bar x_1+ty_1,\dots,s_d\bar x_d+ty_d)\in\Xh\qquad\text{as }\quad0\leq
 t\leq \bar t,\,1-\d\leq s_i\leq1+\d.
 \]
\end{enumerate}
Then
\[
\nabla I(\bar x_1,\dots,\bar x_d)\cdot(y_1,\dots,y_d)\geq0.
\]
\end{lemma}
\begin{proof}
For easier notation we set $\bar x=(\bar x_1,\dots,\bar x_d)$. Since
$(s_1\bar x_1+ty_1,\dots,s_d\bar x_d+ty_d)\in\Xh$ we can substitute
it in each $\bar\l_i$. To start with, we want to apply the Implicit
Function Theorem to
\[
F(s_1,\dots,s_d,t)= \left(
  \begin{array}{c}
    \bar\l_1(s_1\bar x_1+ty_1,\dots,s_d\bar x_d+ty_d) \\
    \vdots \\
    \bar\l_d(s_1\bar x_1+ty_1,\dots,s_d\bar x_d+ty_d) \\
  \end{array}
\right) = \left(
  \begin{array}{c}
    1 \\
    \vdots \\
    1 \\
  \end{array}
\right),
\]
in order to write $s_i=s_i(t)$, where $s_i$ is $C^1$, for every $i$.
By assumption $F$ is $C^1$ and $F(1,\dots,1,0)=(1,\dots,1)$, thus we
only have to prove that the $d\times d$ jacobian matrix
\[
A=\partial_{(s_1,\dots,s_d)}F(1,\dots,1,0)= \left(
    \partial_{x_i}\bar\l_j(\bar x)[\bar x_i]
\right)_{i,j=1,\dots,d} \qquad\text{is invertible}.
\]
Therefore let us assume, by contradiction, the existence of a vector
\begin{equation}\label{eq:s_impl_def}
v=(v_1,\dots,v_d)\in\R^d\setminus\{0\}\text{ such that }Av=0.
\end{equation}
Let us now consider the function
$\Phi(\l_1,\dots,\l_d)=I(\l_1x_1,\dots\l_dx_d)$; by definition, the
point $(\bar\l_1(x_1,\dots,x_d),\dots\bar\l_d(x_1,\dots,x_d))$ is a
free maximum of $\Phi$, and hence
$\nabla\Phi(\bar\l_1,\dots,\bar\l_d)=(0,\dots,0)$, that is
\[
\partial_{x_i}I(\l_1(x_1,\dots,x_d)x_1,\dots,\l_d(x_1,\dots,x_d)x_d)[x_i]=0
\qquad\text{for every }i
\]
(in particular, $\partial_{x_i}I(\bar x)[\bar x_i]=0$ for every
$i$). We can differentiate the previous equation with respect to
$x_j$, obtaining, for every $(z_1,\dots,z_d)\in H^d$,
\[
\partial^2_{x_ix_j}I(\l_1x_1,\dots,\l_dx_d)
[x_i,\l_jz_j]+\sum_{n=1}^d\partial^2_{x_ix_n}
I(\l_1x_1,\dots,\l_dx_d)[x_i,x_n] \cdot\partial_{x_j}\bar\l_n[z_j]=0
\qquad\text{for }j\neq i
\]
and
\begin{multline*}
\partial^2_{x_ix_i}I(\l_1x_1,\dots,\l_dx_d)
[x_i,\l_iz_i]+\sum_{n=1}^d\partial^2_{x_ix_n}
I(\l_1x_1,\dots,\l_dx_d)[x_i,x_n]
\cdot\partial_{x_i}\bar\l_n[z_i]+\\
+\partial_{x_i}I(\l_1x_1,\dots,\l_dx_d)[z_i]=0.
\end{multline*}
We can substitute $x_i=\bar x_i$, $z_i=v_i\bar x_i$, and
$\bar\l_i=1$ in the previous equations. Recalling that
$\partial_{x_i}I(\bar x)[\bar x_i]=0$ for every $i$ we obtain, for
every $i$ and $j$ (not necessarily different),
\[
\sum_{n=1}^d\partial^2_{x_ix_n} I(\bar x)[\bar x_i,\bar x_n]
\cdot\partial_{x_j}\bar\l_n(\bar x)[\bar x_j]\cdot
v_j=-v_j\partial^2_{x_ix_j}I(\bar x) [\bar x_i,\bar x_j].
\]
Summing up on $j$, and recalling the definitions of $A$, $v$
(equation \eqref{eq:s_impl_def}) and $H$, (second assumption of the
lemma) we have
\[
HAv=-Hv,
\]
providing a contradiction with the invertibility of $H$.

Hence we obtain the existence of the $C^1$--functions $s_i=s_i(t)$
(for $t$ sufficiently small) such that $\l_i(s_1\bar
x_1+ty_1,\dots,s_d\bar x_d+ty_d)=1$. Let us consider the function
\[
\varphi(t)=I(s_1(t)\bar x_1+ty_1,\dots,s_d(t)\bar x_d+ty_d).
\]
By construction $\varphi$ is $C^1$ and $\varphi(t)\geq0$ for
$t\geq0$. We obtain that
\[
0\leq\varphi'(0)=\sum_{i=1}^d\partial_{x_i}I(\bar x)[s'_i(t)\bar
x_i+y_i]=\nabla I(\bar
x)\cdot(y_1,\dots,y_d)+\sum_{i=1}^ds'_i(t)\partial_{x_i}I(\bar
x)[\bar x_i],
\]
and the result follows recalling again that $\partial_{x_i}I(\bar
x)[\bar x_i]=0$ for every $i$.
\end{proof}
Now we are finally ready to prove our main results.
\begin{proof}[Proof of Theorem \ref{teo:main}]
Let $\bar\ep$ as above and $\bar\b=\b_2$, in such a way that, for
any fixed $\ep<\bar\ep$ and $\b>\bar\b$, all the previous results
hold. By Remark \ref{rem:equiv_char}, for every integer $s$ we have
an element $(u^s_{11},\dots,u^s_{kh_k})\in\Neh_\b$ such that
\[
c_{\ep,\b}\leq J_\b(U_1^s,\dots,U_k^s)\leq c_{\ep,\b}+\frac{1}{s}.
\]
We are in a situation very similar to that in Lemma
\ref{lem:weak_closure_infty} (much easier, in fact, since now $\b$
is fixed). Following the same scheme, one can easily prove that
$u^s_{im}\to u_{im}^*$ strongly in $\spc$, with
\[
(u^*_{11},\dots,u^*_{kh_k})\in\Neh_\b, \quad
J_\b(U^*_1,\dots,U^*_k)=c_{\ep,\b}.
\]
Moreover, by Lemma \ref{lem:excision}, the minimum point is
$\ep/2$--near an element of $\Kah$.

It remains to prove that each $U^*_i$ is strictly positive and that
$(U^*_1,\dots,U^*_k)$ solves \eqref{eq:sys}. To do this we will
apply Lemma \ref{lem:variational_lemma}, letting $H=\spc$, $d=h$,
$\Xh=\tilde\Xh_\ep$, $I(u_{11},\dots,u_{kh_k})=
J_\b(U_1,\dots,U_k)$, and $(\bar x_1,\dots,\bar
x_h)=(u^*_{11},\dots,u^*_{kh_k})$. Assumptions 1. and 2. in Lemma
\ref{lem:variational_lemma} are satisfied by construction, therefore
we have only to choose a variation $(y_1,\dots,y_h)$ and to check
assumption 3.:
\[
\text{``}P=(s_{11}u^*_{11}+t y_1,\dots, s_{kh_k}u^*_{kh_k}+t
y_h)\in\tilde\Xh_\ep\text{ when }t>0\text{ is
 small and each }s_{im}\text{ is near }1\text{''}.
\]
Under these assumptions on $t$ and $s_{im}$, it is immediate to see
that $P$ is $\ep$--near to the same element of $\Kah$ to which
$(u^*_{11},\dots,u^*_{kh_k})$ is $\ep/2$--near; moreover, by
continuity, $M_\b(P)=J_\b(\bar\L(P)P)<c_{\ep,\b}+1/\b$. Recalling
the definition of $\tilde X_\ep$ (Corollary \ref{coro:tilde_X}) we
have that assumption 3. is fulfilled whenever each component of $P$
is non negative.

First let us prove that each $U_i$ is strictly positive. Assume not,
there exists $x_0\in\R^N$ with, say, $U_1(x_0)=0$. Since
$U_1\not\equiv0$ we can easily construct an open, relatively compact
annulus $A\ni x_0$ such that $U_1\leq1/2$ on $A$ and
$U_1\not\equiv0$ on $\partial A$. For any radial $\f\in
C^{\infty}_0(A)$, $\f\geq0$, we choose the variation $y_1=\f$,
$y_l=0$ for $l>1$. Clearly each component of $P$ is non negative,
thus Lemma \ref{lem:variational_lemma} implies
\[
\begin{split}
0&\leq \nabla
I(u^*_{11},\dots,u^*_{kh_k})\cdot(\f,0,\dots,0)=\partial_{u_{11}}
I(u^*_{11},\dots,u^*_{kh_k})[\f]=\\
&=\int_{A}\left[\nabla
U_1\cdot\nabla\f+U_1\left(1-U_1^2+\b\sum_{j\neq
1}U_j^2\right)\f\right]\,dx\\
&=\int_{A}\left[\nabla
U_1\cdot\nabla\f+a(x)U_1\f\,dx\right]\qquad\text{for every radial
}\f\in C^{\infty}_0(A).
\end{split}
\]
But then, since $a(x)\geq3/4>0$ on $A$, and $U_1\not\equiv0$ on
$\partial A$, the strong maximum principle implies $U_1>0$ on $A$, a
contradiction.

Now let us prove that $(U^*_1,\dots,U^*_k)$ solves \eqref{eq:sys}.
Again, assume by contradiction that, for instance, $U_1$ does not
satisfy the corresponding equation. Then there exists one radial
$\f\in C^\infty_0(\R^N)$, not necessarily positive, such that (up to
a change of sign)
\[
\int_{\R^N}\left(\nabla U_1\cdot\nabla\f+U_1\f-U_1^3\f+\b
U_1\sum_{j\neq 1}U_j^2\f \right)\,dx<0.
\]
Moreover we can choose $\f$ with support arbitrarily small. Since
$U_1$ is strictly positive, we can then assume that one of its
pulse, say $u_{11}$, is strictly positive on the support of $\f$.
But then, for $t$ small, also $s_{11}u^*_{11}+t\f$ is positive,
therefore Lemma \ref{lem:variational_lemma} (with $y_1=\f$, $y_l=0$
for $l>1$) implies
\[
0\leq \nabla
I(u^*_{11},\dots,u^*_{kh_k})\cdot(\f,0,\dots,0)=\int_{\R^N}\left(\nabla
U_1\cdot\nabla\f+U_1\f-U_1^3\f+\b U_1\sum_{j\neq 1}U_j^2\f
\right)\,dx,
\]
a contradiction.
\end{proof}

\begin{proof}[Proof of Theorem \ref{teo:main2}]
The proof readily follows by proving that
\begin{center}
for every $0<\nu<1$, if $\beta$ is sufficiently large, then
$\Neh_\b\subset\tilde\Xh_{\nu\ep}$.
\end{center}
But this can be done following the line of the proof of Lemma
\ref{lem:excision}.
\end{proof}

\begin{rem}\label{rem:finale}
We proved the main result in the simplest case of system
\eqref{eq:sys}. Now we suggest how to modify this scheme in order to
treat the general case of system \eqref{eq:sys_comp}. The main
difference is that the role of the associated limiting equation is
now played by the minimization problem
\[
\min_{\Xh^*}\sum_{l=1}^h\dfrac{\int_{\R^N}|\nabla
w_l|^2+\left(V_{\sigma(l)}(x)+\l_{\sigma(l)}\right)w_l^2\,dx}{
\left(\int_{\R^N}\mu_{\sigma(l)}w_l^4\,dx\right)^{1/2}},
\]
where now the constants $\mu_i$'s and $\l_i$'s are allowed to take
different values, and also the potentials $V_i$'s, with the only
constraint that each Schr\"odinger operator
\[
 -\Delta +V_i(x)+\l_i
\]
must be positive. In such a situation, the above minimization
problem is always solvable and we call $\Kah$ its solution set. With
these changes , in dimensions two and three, Theorem \ref{teo:main}
and all its proof remain the same.

When $\lim_{|x|\to+\infty}V_i(x)=+\infty$ we can lower the dimension
to cover also the dimension $N=1$. In such a case we must change the
definition of the Hilbert space we work in choosing for each $i$ the
different norm
\[
\|U_i\|^2=\int_{\R^N}|\nabla
U_i|^2+\left(V_i(x)+\l_i\right)U_i^2\,dx,
\]
with the advantage that the embedding in $L^4$ is now compact also
in dimension $N=1$.

Finally, let us mention that we can also allow bounded radially
symmetric domains instead of $\R^N$, and some non--cubic
nonlinearities, provided they are subcritical.
\end{rem}


\end{document}